\documentclass{amsart}

\usepackage{graphicx,afterpage,float,lscape, caption}

\usepackage{amsmath}
\usepackage{amsfonts}
\usepackage{amssymb}
\usepackage{amsthm}
\usepackage{verbatim}
\usepackage{hyperref}
\usepackage{mathrsfs}
\usepackage{wrapfig}
\usepackage{tikz}
\usetikzlibrary{matrix,arrows,decorations.pathmorphing}
\usepackage{url}

\newcommand{\bb}[1]{\mathbb{#1}}

\newcommand{\F}[1]{\bb{F}_{#1}}

\newcommand{\pomegap}[2]{{\sf P}\Omega^+_{#1}(#2)}
\newcommand{\grid}[2]{(#1 \times #2){\sf -grid}}
\newcommand{\gridcomp}[1]{\overline{\grid{2}{#1}}}

\DeclareMathOperator{\VGamma}{V\Gamma}

\DeclareMathOperator{\PSL}{PSL}
\DeclareMathOperator{\Sp}{Sp}

\DeclareMathOperator{\Soc}{Soc}

\DeclareMathOperator{\K}{K}

\DeclareMathOperator{\Aut}{Aut}

\DeclareMathOperator{\Hamming}{H}

\DeclareMathOperator{\diam}{diam}
\newtheorem{theorem}{Theorem}[section]
\newtheorem{corollary}[theorem]{Corollary}

\newtheorem{lemma}[theorem]{Lemma}
\newtheorem{proposition}[theorem]{Proposition}

\theoremstyle{definition}
\newtheorem{definition}[theorem]{Definition}

\newtheorem{condition}[theorem]{Condition}

\renewcommand{\leq}{\leqslant}
\renewcommand{\geq}{\geqslant}

\usepackage{array}
\title{Finite 2-distance transitive graphs}

\author[B.~Corr]{Brian P. Corr}

\author{Wei Jin}
\author{Csaba Schneider}
\address[B.\ Corr, W.\ Jin and C.\ Schneider]
{Departamento de Matem\'{a}tica\\
Instituto de Ci\^{e}ncias Exatas\\
Universidade Federal de Minas Gerais\\
Av. Ant\^{o}nio Carlos, 6627, 31270-901\\
Belo Horizonte, MG, Brazil}

\address[W.\ Jin]{School of Statistics\\
  Jiangxi University of Finance and Economics\\
 Nanchang, Jiangxi, 330013, P.R.China}
\address[W.\ Jin]{Research Center of Applied Statistics\\
  Jiangxi University of Finance and Economics\\
  Nanchang, Jiangxi, 330013, P.R.China}
\email{brian.p.corr@gmail.com, csaba@mat.ufmg.br, jinweipei82@163.com}


\begin{document}
\maketitle

\begin{abstract}
A  non-complete graph    $\Gamma$ is said to be $(G,2)$-distance
transitive if $G$ is a subgroup of the automorphism group of $\Gamma$ that
is transitive on the vertex set of $\Gamma$, and
for any vertex $u$ of $\Gamma$, the stabilizer $G_u$ is transitive
on the sets of vertices at distance~1 and~2 from $u$. This paper
investigates the family of $(G,2)$-distance transitive graphs that
are not $(G,2)$-arc transitive. Our main result is the classification
of such graphs of valency not greater than~5.
\end{abstract}

\section{Introduction}

Graphs that satisfy certain symmetry conditions have been a focus of research
in algebraic graph theory. We usually measure the degree of symmetry of
a graph by studying if the automorphism group is transitive on certain
natural sets formed by combining vertices and edges.
For instance,
$s$-arc transitivity requires that the automorphism group should be
transitive on the set of $s$-arcs (see Section~\ref{sect:def} for precise definitions).  The class of
$s$-arc transitive graphs have been  studied intensively,
beginning with the seminal result of Tutte
\cite{Tutte-1} that cubic $s$-arc transitive graphs must
have $s\leq 5$.  Later, in 1981, Weiss \cite{weiss}, using the
finite simple group classification, showed that  there are no
$8$-arc transitive graphs of valency at least 3. For a survey on
$s$-arc transitive graphs, see~\cite{seress}.

Recently, several papers have considered conditions on undirected
graphs that are similar to, but weaker than,
$s$-arc transitivity. For examples of such conditions, we mention local $s$-arc transitivity, 
local $s$-distance transitivity, $s$-geodesic transitivity, and
$2$-path transitivity.
Devillers et al.~\cite{locallysdist}
studied the class of locally $s$-distance transitive graphs, 
using
the normal quotient strategy
developed for $s$-arc transitive graphs in~\cite{Praeger-4}.
The condition of $s$-geodesic transitivity was investigated in
several papers~\cite{DJLP-2,DJLP-prime,DJLP-compare}.
A characterization of $2$-path transitive, but not
$2$-arc transitive graphs was given by  Li and Zhang~\cite{LZ-2path-2013}.

In this paper we study the class of $2$-distance transitive graphs.
If $G$ is a subgroup of the automorphism group of a graph $\Gamma$, then
$\Gamma$ is said to be $(G,2)$-distance transitive if $G$ acts
transitively on the vertex set of $\Gamma$, and a vertex stabilizer
$G_u$ is transitive on the neighborhood $\Gamma(u)$ of $u$ and on the second
neighborhood $\Gamma_2(u)$ (see Section~\ref{sect:def}).
The class of $(G,2)$-distance transitive graphs is larger than
the class of $(G,2)$-arc transitive graphs, and in this paper we
study the $(G,2)$-distance transitive graphs that are not $(G,2)$-arc
transitive.

Our first theorem links the structure of $(G,2)$-distance transitive, but not
$(G,2)$-arc transitive graphs to their valency and the
value of the constant $c_2$ in the intersection array (see Definition
\ref{definition: intersectionarray}).

\begin{theorem}\label{thm:valency stuff}
  Let $\Gamma$ be a connected $(G,2)$-distance transitive,
  but not $(G,2)$-arc transitive graph of girth $4$ and valency $k\geq 3$.
  Then $2\leq c_2\leq k-1$ and the following are valid.
  \begin{enumerate}
  \item  If $c_2=k-1$, then $\Gamma \cong \gridcomp{(k+1)}$ and
    $G$ satisfies Condition~\ref{gridcondition}.
\item  If $c_2=2$, then $k$ is a prime-power such
  that $k\equiv 3 \pmod 4$ and $G_u$ acts $2$-homogeneously,
  but not $2$-transitively on $\Gamma(u)$ for each $u\in V\Gamma$.
  \end{enumerate}

\end{theorem}

The following corollary is a characterization of the family of connected $(G,2)$-distance transitive, but not $(G,2)$-arc transitive graphs of  girth $4$ and prime valency.

\begin{corollary}\label{thm:primeval}
  Let $\Gamma$ be a connected $(G,2)$-distance transitive, but not $(G,2)$-arc transitive graph of  girth $4$ and prime valency $p$, and let $u\in V \Gamma$.
  Then the following are valid.
  \begin{enumerate}
  \item  Either $\Gamma\cong \overline{\grid 2{(p+1)}}$, or    $c_2 |p-1$ and $2\leq c_2\leq (p-1)/2$.
    \item 
      If $ c_2=2$, then $p\equiv 3 \pmod 4$ and $G_u$ is $2$-homogeneous, but not $2$-transitive  on $\Gamma(u)$.
      \item If $ c_2=(p-1)/2$, then $|\Gamma_2(u)|=2p$, and  $G_u$
        is imprimitive on  $\Gamma_2(u)$.
        \end{enumerate}
\end{corollary}

Finally, our third main result  determines all the possible
$(G,2)$-distance transitive, but not $(G,2)$-arc transitive
graphs of  valency at most 5.

\begin{theorem}\label{thm:small val}
Let $\Gamma$ be a connected $(G,2)$-distance transitive, but not $(G,2)$-arc transitive graph of valency $k\leq 5$. Then $\Gamma$ and $G$ must be as in one of
the rows of Table~\ref{maintable}.
\end{theorem}









{
\begin{table}
\begin{center}
\begin{tabular}{|l|c|c|l|l|}
\hline
$\Gamma$ & valency &  girth & $G$   & Reference \\
\hline 
$\gridcomp{4}$ & $3$ & $4$ & satisfies Condition~\ref{gridcondition} & Section \ref{sec:gridcomp} \\
\hline
Octahedron & $4$ & $3$ &
{\setlength\extrarowheight{0pt}\begin{tabular}{l}
$G\leq S_2\wr S_3$,\\ $|S_2\wr S_3:G|\in\{1,2\}$,\\
$G$ projects onto $S_3$
\end{tabular}}  & Lemma \ref{lem:Octahedron} \\
\hline
$\Hamming(2,3)$ & $4$ & $3$ &
{\setlength\extrarowheight{0pt}\begin{tabular}{l}
$G\leq S_3\wr S_2$,\\ $|S_3\wr S_2:G|\in\{1,2\}$\\
$G$ projects onto $S_2$
\end{tabular}} & Proposition \ref{2dtval4-girth3}\\
\hline
{\setlength\extrarowheight{0pt}\begin{tabular}{l}
the line graph of a connected\\
$(G,3)$-arc transitive graph
\end{tabular}} & 4 & 3 & & Proposition \ref{2dtval4-girth3} \\
\hline
$\gridcomp 5$ & $4$ & $4$ & satisfies Condition~\ref{gridcondition} & Section \ref{sec:gridcomp}\\
\hline
Icosahedron & $5$ & $3$ & $G\in\{A_5,A_5\times C_2\}$ & Lemma \ref{lem:Icosahedron}\\
\hline
$\gridcomp 6$ & $5$ & $4$ & satisfies Condition~\ref{gridcondition} & Section \ref{sec:gridcomp}\\
\hline
\end{tabular}
\end{center}
\caption{$(G,2)$-distance transitive, but not $(G,2)$-arc transitive
graphs of valency at most 5}
\label{maintable}
\end{table}}

In Section~\ref{sect:def} we state the most important definitions and some
basic results related to
$2$-distance transitivity. In Section~\ref{sect:exam},
we study some examples, such
as grids, their complements, Hamming graphs, complete bipartite graphs, and
platonic solids from the point of view of $2$-distance transitivity.
In Section~\ref{sect:girth4}, we consider $2$-distance transitive graphs of girth 4.
Finally the proofs of our main results are given in Section~\ref{sect:proofs}.

\subsection*{Acknowledgment}
During the preparation of this article, the first and second authors
held {\em Ci\^encia sem Fronteiras/Jovem Talento} and
{\em Programa Nacional de P\'os-doutorado} fellowships, respectively,
awarded by the Brazilian Government. In addition, the second
author was also awarded the NNSF grant  11301230 (China).
The third author was supported by the
 research projects 302660/2013-5 {\em (CNPq, Produtividade em Pesquisa)},
 475399/2013-7 {\em (CNPq, Universal)}, and 
 the APQ-00452-13 {\em (Fapemig, Universal)}.

\section{Basic Definitions and useful facts}\label{sect:def}

In this paper,  graphs are finite, simple, and undirected. For a graph $\Gamma$, let $\VGamma$ and $\Aut\Gamma$ denote its vertex
set and automorphism group, respectively.
Let $\Gamma$ be a graph and let
$u$ and $v$ be vertices in $\Gamma$
that belong to the same connected component.
Then the {\em distance} between $u$ and $v$ is the length
of a shortest path between $u$ and $v$  and is
denoted by $d_{\Gamma}(u,v)$. We denote by $\Gamma_s(u)$ the set of vertices at
distance $s$ from $u$ in $\Gamma$ and we set $\Gamma(u)=\Gamma_1(u)$.
The \emph{diameter} $\diam\Gamma$
of $\Gamma$ is the greatest distance between vertices in $\Gamma$.
Let  $G\leq\Aut\Gamma$ and let $s\leq\diam\Gamma$.
We say that $\Gamma$ is \emph{$(G,s)$-distance transitive}
if $G$ is transitive on $\VGamma$ and
$G_u$ is transitive on $\Gamma_i(u)$ for all $i\leq s$.
If $\Gamma$ is $(G,s)$-distance transitive for
$s= \diam \Gamma$,
then we simply say that it is \emph{$G$-distance transitive}.
By our definition, if $s>\diam\Gamma$, then $\Gamma$ is not
$(G,s)$-distance transitive. For instance, the complete graph is not
$(G,2)$-distance transitive for any group $G$.

 In the characterization of $(G,s)$-distance transitive
graphs, the following constants are useful. Our definition is inspired by the concept of intersection arrays defined 
for the distance regular graphs (see \cite{BCN}).

\begin{definition}\label{definition: intersectionarray}
Let $\Gamma$ be a $(G,s)$-distance transitive graph, $u\in\VGamma$, and
let $v\in\Gamma_i(u)$, $i\leq s$. Then the number of edges
from $v$ to $\Gamma_{i-1}(u)$, $\Gamma_i(u)$, and $\Gamma_{i+1}(u)$
does not depend on the choice of $v$ and these numbers are 
denoted, respectively, by $c_i$, $a_i$, $b_i$.
\end{definition}

Clearly we have that $a_i+b_i+c_i$ is equal to the valency of $\Gamma$ whenever the constants are well-defined. Note that for $(G,2)$-distance transitive graphs, the constants are always well-defined for $i=1,\ 2$.


A sequence $(v_0,\ldots,v_{s})$ of vertices of a graph is said to be
an {\em $s$-arc} if $v_i$ is connected to $v_{i+1}$
for all $i\in\{0,\ldots,s-1\}$ and $v_i\neq v_{i+2}$
for all $i\in\{0,\ldots,s-2\}$.
A graph $\Gamma$ is called
\emph{$(G,s)$-arc transitive} if $G$ acts transitively on the set
of vertices and on the set of $s$-arcs of $\Gamma$. (We note that some
authors define $(G,s)$-arc transitivity only requiring that
$G$ should be transitive on the set of $s$-arcs.)
It is well-known, that $\Gamma$ is $(G,2)$-arc transitive
if and only if $G$ is transitive on $\VGamma$, and
the stabilizer $G_u$ is 2-transitive on
$\Gamma(u)$ for some, and hence for all, $u\in\VGamma$. We will use this fact
without further reference in the rest of the paper.


%
The {\em girth} of a graph $\Gamma$ is the length of a shortest cycle
in $\Gamma$. Let $\Gamma$ be a connected  $(G,2)$-distance transitive
graph. If  $\Gamma$ has girth at least 5, then for any two vertices
$u$ and $v$ with $d_{\Gamma}(u,v)=2$, there exists a unique 2-arc between
$u$ and $v$. Hence if $\Gamma$ is $(G,2)$-distance transitive, then it is
$(G,2)$-arc transitive. On the other hand, if the girth of $\Gamma$ is 3, and
$\Gamma$ is not a complete graph, then some $2$-arcs are contained
in a triangle, while some are not. Hence $\Gamma$ is not $(G,2)$-arc transitive.
We record the conclusion of this argument  in the following lemma.

\begin{lemma}\label{lem:girth5}
  Suppose that $\Gamma$ is a $(G,2)$-distance transitive graph. If $\Gamma$ has girth at least $5$, then $\Gamma$ is $(G,2)$-arc transitive. If $\Gamma$ has
  girth~$3$, then $\Gamma$ is not $(G,2)$-arc transitive.
\end{lemma}

If  $\Gamma$ has girth 4, then $\Gamma$ can be $(G,2)$-distance
transitive, but not $(G,2)$-arc transitive.
An infinite family of examples can be constructed using Lemma~\ref{lem:gridcomp}.

We close this section with two results on permutation group
theory and another one on
$2$-geodesic transitive graphs.
They will be needed in our analysis
in Sections~\ref{sect:girth4}--\ref{sect:proofs}.
Recall that a permutation group $G$ acting on $\Omega$ is said to be
{\em $2$-homogeneous} if $G$ is transitive on the set of $2$-subsets of
$\Omega$.

\begin{lemma}[\cite{kantor}] \label{2dt-2homonot2t}
 Let $G$ be a  $2$-homogeneous permutation group of
   degree $n$ which is not $2$-transitive. Then
      the following statements are valid:
\begin{enumerate}
\item $n=p^e\equiv 3 \pmod 4$ where $p$ is a prime;
\item $|G|$ is odd and is divisible by $p^e(p^e-1)/2$; 
\end{enumerate}
\end{lemma}

\begin{lemma}{\rm(\cite[Theorem 1.51]{Gorenstein-1})}\label{val-2p-1}
If $G$ is a primitive, but not $2$-transitive permutation group on
$2p$ letters where $p$ is a prime, then $p=5$ and $G\cong A_5$ or $S_5$.
\end{lemma}

An \emph{$s$-geodesic} in a graph $\Gamma$ is a shortest path of length $s$ between vertices in $\Gamma$. In particular, a vertex triple $(u,v,w)$  with $v$ adjacent to both $u$ and $w$ is called a
\emph{$2$-geodesic} if  $u$ and $w$ are not adjacent.  A non-complete
graph $\Gamma$   is said to be  \emph{ $(G,2)$-geodesic transitive}
if $G$  is transitive on
both the arc set and  on the
set of 2-geodesics of $\Gamma$.
Recall that the {\em line graph} $L(\Gamma)$
of a graph $\Gamma$ is graph whose vertices are the edges of
$\Gamma$ and two vertices of $L(\Gamma)$ are adjacent if and only if
they are adjacent to a common vertex of $\Gamma$. For a natural number $n$,
we denote by $\K_n$ the {\em complete graph} on $n$ vertices.


\begin{lemma}{\rm (\cite[Theorem 1.3]{DJLP-2})}\label{2gt-val4}
  Let $\Gamma$ be a connected, non-complete graph of valency  $4$ and girth $3$.
  Then  $\Gamma$ is $(G,2)$-geodesic transitive if and
  only if, either  $\Gamma=L(\K_4)$
  or $\Gamma=L(\Sigma)$ where $\Sigma$ is connected
  cubic $(G,3)$-arc transitive graph.
\end{lemma}

We observe that the line graph of $\K_4$ is precisely the octahedral graph (see Lemma \ref{lem:Octahedron}).

\section{Constructions, Examples \& non-Examples}\label{sect:exam}

\subsection{Complements of grids and complete bipartite graphs}\label{sec:gridcomp}


For $n,\ m\geq 2$, we define the $\grid{n}{m}$ as the graph having vertex set $\{(i,j) \mid 1 \leq i\leq n,\ 1\leq j \leq m\}$, and two distinct vertices $(i,j)$ and $(r,s)$
are adjacent if and only if $i=r$ or $j=s$.
The automorphism group of the $\grid{n}{m}$, when $n \neq m$, is the direct product $S_n \times S_m$; when $n=m$, it is $S_n \wr S_2$.
The   complement $\overline{\Gamma}$ of a graph $\Gamma$, is the graph with vertex set $V\Gamma$, and two vertices are adjacent in $\overline{\Gamma}$ if and only if they are not adjacent in $\Gamma$. Clearly, $\Aut\Gamma=\Aut\overline\Gamma$. 
Of particular interest to us is the complement graph $\gridcomp{m}$. The graph in Figure \ref{fig:gridcomp} is the $\gridcomp{4}$. Observe  that for $\Gamma=\gridcomp{m}$, we have $\diam\Gamma = 3$, and
\[
c_1 = 1,\ a_1 = 0,\ b_1 = m-2,\ c_2 = m-2,\ a_2 = 0,\ b_2 = 1.
\]

\begin{figure}[h]

\begin{tikzpicture}


\draw (0,2) -- (2,3);
\draw (0,2) -- (2,4);
\draw (0,2) -- (2,5);

\draw (0,3) -- (2,2);
\draw (0,3) -- (2,4);
\draw (0,3) -- (2,5);

\draw (0,4) -- (2,2);
\draw (0,4) -- (2,3);
\draw (0,4) -- (2,5);

\draw (0,5) -- (2,2);
\draw (0,5) -- (2,3);
\draw (0,5) -- (2,4);



\filldraw[white] (0,2) circle (2pt);
\draw[black] (0,2) circle (2pt) node[left]{$(1,1)$};

\filldraw[white] (0,3) circle (2pt);
\draw[black] (0,3) circle (2pt) node[left]{$(1,2)$};

\filldraw[white] (0,4) circle (2pt);
\draw[black] (0,4) circle (2pt) node[left]{$(1,3)$};

\filldraw[white] (0,5) circle (2pt);
\draw[black] (0,5) circle (2pt) node[left]{$(1,4)$};



\filldraw[white] (2,2) circle (2pt);
\draw[black] (2,2) circle (2pt) node[right]{$(2,1)$};

\filldraw[white] (2,3) circle (2pt);
\draw[black] (2,3) circle (2pt) node[right]{$(2,2)$};

\filldraw[white] (2,4) circle (2pt);
\draw[black] (2,4) circle (2pt) node[right]{$(2,3)$};

\filldraw[white] (2,5) circle (2pt);
\draw[black] (2,5) circle (2pt) node[right]{$(2,4)$};



\draw (6,2.5) -- (7,3.5);
\draw (6,2.5) -- (7,4.5);

\draw (6,3.5) -- (7,2.5);
\draw (6,3.5) -- (7,4.5);

\draw (6,4.5) -- (7,2.5);
\draw (6,4.5) -- (7,3.5);


\draw(5,3.5) -- (6,2.5);
\draw(5,3.5) -- (6,3.5);
\draw(5,3.5) -- (6,4.5);

\draw(8,3.5) -- (7,2.5);
\draw(8,3.5) -- (7,3.5);
\draw(8,3.5) -- (7,4.5);

\filldraw[white] \foreach \x in {
(7,2.5),
(7,3.5),
(7,4.5),
(6,2.5),
(6,3.5),
(6,4.5),
(5,3.5),
(8,3.5)
}
{
\x circle(2pt)
};
\draw[black] \foreach \x in {
(7,2.5),
(7,3.5),
(7,4.5),
(6,2.5),
(6,3.5),
(6,4.5),
(5,3.5),
(8,3.5)
}
{
\x circle(2pt)
};

\draw[black] (5,3.5) circle (2pt) node[left]{$(1,1)$};
\draw[black] (8,3.5) circle (2pt) node[right]{$(2,1)$};

\end{tikzpicture}
\caption{The grid complement $\gridcomp{4}$; and on the right represented according to a distance-partition.}\label{fig:gridcomp}
\end{figure}
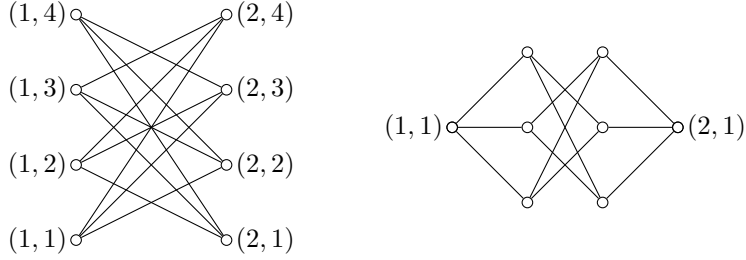

\begin{condition}\label{gridcondition}
Let $m\geq 3$ and let $\pi: S_2\times S_m\rightarrow 
S_2$ be the natural projection. We say that a subgroup  $G$ of 
$S_2\times S_m$ satisfies Condition~\ref{gridcondition} if
$G\pi=S_2$ and  
$G\cap S_m$ is a  $2$-transitive, but not $3$-transitive subgroup of $S_m$.
\end{condition}

\begin{lemma}\label{lem:gridcomp}
Let $\Gamma = \gridcomp{m}$ with $m\geq 4$, and let $G\leqslant \Aut\Gamma =S_2\times S_m$. Then $\Gamma$ is  $(G,2)$-distance transitive, but not $(G,2)$-arc transitive if and only if $G$ satisfies Condition \ref{gridcondition}.
\end{lemma}
\begin{proof}
Let $\Delta_1=\{(1,i)\mid i=1,2,\ldots,m\}$ and
$\Delta_2=\{(2,i)\mid i=1,2,\ldots,m\}$ be the two biparts of
$V\Gamma$. Let $u=(1,1)\in \Delta_1$.  Suppose first that $\Gamma$ is
$(G,2)$-distance transitive, but not $(G,2)$-arc transitive.  Since
$G$ is transitive on $V\Gamma$, $G$ projects onto $S_2$, that is,
$G\pi=S_2$. Let $H=G\cap S_m$. Then $G_u=H_1$, $\Delta_2=\Gamma(u)\cup
\{(2,1)\}$ and $\Gamma_2(u)=\Delta_1\setminus \{u\}$. Since $\Gamma$
is  $(G,2)$-distance transitive, $G_u=H_1$ is transitive on both
$\Gamma(u)$ and $\Gamma_2(u)$. Hence $H_1$ is transitive on $\{2,\ldots,m\}$, 
and so $H$ is a  $2$-transitive  subgroup
of $S_m$. Since $\Gamma$ is   not $(G,2)$-arc transitive,
$G_u=H_1$ is not 2-transitive on  $\{2,3,\ldots,m\}$,
so $H$ is not  $3$-transitive. Thus $G$ satisfies Condition
\ref{gridcondition}.

Conversely, suppose that $G$ satisfies Condition \ref{gridcondition}.
Then  $H=G\cap S_m$ is transitive on $\Delta_1$ and $\Delta_2$,
and $G$ swaps these two sets. Thus $G$ is transitive on $\VGamma$.
As $H$ is a 2-transitive, but not 
3-transitive subgroup of $S_m$, $H_1$ is transitive, but not 2-transitive on 
$\Gamma(u)=\{(2,i)\mid i=2,\ldots,m\}$ 
and on $\Gamma_2(u)=\{(1,i)\mid i=2,\ldots,m\}$. 
Hence $\Gamma$ is  $(G,2)$-distance transitive, but not $(G,2)$-arc transitive.
\end{proof}

A list of $2$-transitive, but not $3$-transitive permutation groups 
can  be found in~\cite[pp.~194-197]{cameron}.

Complete bipartite graphs   appear frequently in this paper. Since  $\K_{m,n}$ with $m\neq n$ is not regular, we study $\K_{m,m}$.
The full automorphism group of $\K_{m,m}$ is $S_m \wr S_2$, and this
automorphism group acts $2$-arc transitively on $\K_{m,m}$. In the lemma below, we show that there is no $2$-distance transitive
 action on $\K_{m,m}$ which is not $2$-arc transitive.


\begin{lemma}\label{lem:complete bipartite}
Let $\Gamma\cong \K_{m,m}$ with $m\geq 2$ and let
$G\leq\Aut \Gamma$.
Then  $\Gamma$  is $(G,2)$-distance transitive if and only if it  is $(G,2)$-arc transitive.
\end{lemma}

\begin{proof}[Proof]
If $\Gamma$ is $(G,2)$-arc transitive, then, 
by definition, it  is $(G,2)$-distance transitive.
Conversely, suppose that $\Gamma$ is $(G,2)$-distance transitive
with some $G\leq\Aut \Gamma$. 
Let $\VGamma=\Delta_1\cup\Delta_2$ be the bipartition of $V \Gamma$ where
$\Delta_1=\{(1,i)\mid i=1,\ldots,m\}$ and $\Delta_2=\{(2,i)\mid i=1,\ldots,m\}$.
The full automorphism group of $\Gamma$ is $S_m\wr S_2$.
Since $G\leq\Aut \Gamma$ is assumed to be vertex transitive,
 $G_{\Delta_1}=G_{\Delta_2}$ is transitive
on both $\Delta_1$ and $\Delta_2$. Set $G_0=G_{\Delta_1}$. Thus $G_0$ is a
subdirect subgroup in $M^{(1)}\times M^{(2)}$ where $M^{(i)}\leq S_m$
and $M^{(i)}$ is the image of $G_0$
under the $i$-th coordinate projection $S_m\times
S_m\rightarrow S_m$.
Further, $G$ projects
onto $S_2$ under the natural projection $\Aut \Gamma\rightarrow S_2$.
If $x=(x_1,x_2)\sigma\in G$ with $x_i\in S_m$ and $\sigma=(1,2)\in S_2$, then
$(M^{(1)})^{x_1}=M^{(2)}$, and so   $M^{(1)}$ and $M^{(2)}$ are conjugate 
subgroups of $S_m$. Hence possibly replacing $G$ with its conjugate
$G^{(x_1,1)}$, we may assume without loss of generality that $M^{(1)}=M^{(2)}=M$.

Let $u=(1,1)\in\VGamma$. 
Then $\Gamma(u)=\Delta_2$ and $\Gamma_2(u)=\Delta_1\setminus \{u\}$.
Further, $G_u$ stabilizes $\Delta_1$, and hence $G_u\leq G_0$.
Since $\Gamma$  is $(G,2)$-distance transitive, it follows that $G_u$ is
transitive on both $\Delta_2$ and $\Delta_1\setminus \{u\}$.
Set $H=M_1$.
Since $G_u\leq H\times M$, the stabilizer $H$ must be transitive on 
$\{2,\ldots,m\}$, and 
hence $M$ is a 2-transitive subgroup of $S_m$.
In particular $M$ contains a unique minimal
normal subgroups $N$ and this minimal normal subgroup is either
elementary abelian or simple. 
Since $N$ is transitive, we can write $M=NH$.
We have that $G_0$ contains $1\times N$ 
if and only if it contains $N\times 1$. 
Hence we need to consider two cases: the first
is when $G_0$ contains $N\times N$ and the second is when  it
does not.

Suppose first that $G_0$ contains
$N\times N$. In particular, $1\times N\leq G_u$.
For all $h_2\in H$, there is some $n_1h_1\in M$
with $n_1\in N$ and $h_1\in H$ such
that $(n_1h_1,h_2)\in G_0$.
Since $N\times 1\leq G_0$, this implies that
$(h_1,h_2)\in G_0$ and also $(h_1,h_2)\in G_u$.
Thus $G_u$ projects onto  $NH=M$ by the second projection. 
Hence $G_u$ is 2-transitive on $\Delta_2=\Gamma(u)$, which shows
that $\Gamma$ is $(G,2)$-arc transitive. 

Suppose now that $N\times N$ is not contained in $G_0$. 
Since $G_0\cap (1\times M)$ is a normal subgroup of $M$ and $N$ is the unique 
minimal normal  subgroup of $M$, we find that $G_0\cap (1\times M)=1$ 
and, similarly, that $G_0\cap (M\times 1)=1$.
Therefore
$G_0$ is a diagonal subgroup; that is,
$$
G_0=\{(t,\alpha(t))\mid t\in M\}
$$
with some $\alpha\in\Aut M$.
As $H$ is the stabilizer of $1$ in $M$,
we have that $G_u=\{(t,\alpha(t))\mid t\in H\}$.
On the other hand, $G_u$ is transitive on $\Delta_2$, and hence
$\alpha(H)$ is a transitive subgroup of $M$.
Thus we obtain the factorization $M=H\alpha(H)$.
The
following possibilities are listed in~\cite[Theorem~1.1]{baum}.
\begin{enumerate}
\item[(a)] Either $M$ is affine and is isomorphic to
$[(\F{2})^3\rtimes \PSL(3,2)]\wr X$ where $X$ is a transitive permutation group;
\item[(b)]
or $\Soc M\cong \pomegap 8q$, $\Sp(4,q)$ ($q$ even with $q\geq 4$), $A_6$, $M_{12}$.
\end{enumerate}
In case~(a), if $X\neq 1$, then
$M$ is contained in a wreath product in
product action, and such a wreath product is never 2-transitive.
Thus $X=1$, $m=8$, $M=(\F{2})^3\rtimes \PSL(3,2)$, and
$G_u\cong\PSL(3,2)$ acting transitively on $\Delta_2$. However, this transitive
action
of $\PSL(3,2)$ is $2$-transitive, which gives that $\Gamma$ is
$(G,2)$-arc transitive.

In case~(b), inspecting the list of
almost simple 2-transitive groups in~\cite{cameron},
we find that there are no 2-transitive groups with socle
$\pomegap 8q$ or $\Sp(4,q)$ with $q$ even and $q\geq 4$.
Hence $\Soc M= A_6$ or
$M_{12}$. Then $G_u$ is either $A_5$, $S_5$ or $M_{11}$ acting
transitively on $\Delta_2$. These actions are all 2-transitive, which implies
that $\Gamma$ is $(G,2)$-arc transitive.
\end{proof}

\subsection{Hamming graphs and platonic solids}

For $d,\ q\geq 2$, the vertex set of the
{\em Hamming graph} $\Hamming(d,q)$ is the set $\{1,\ldots,q\}^d$ and
two vertices $u=(\alpha_1,\ldots,\alpha_d)$ and $v=(\beta_1,\ldots,\beta_d)$
are adjacent if and only if their Hamming distance is one; that is,
they differ in precisely one coordinate.
The Hamming graph 
has diameter $d$ and has girth $4$ when $q=2$ and girth $3$ when $q > 2$. 
The wreath product $W = S_q\wr S_d$ is the full automorphism group of $\Gamma$, acting distance transitively, see \cite[Section 9.2]{BCN}. The Hamming graphs are well studied, due in part to their applications to coding theory.
Hamming graphs arise in two cases   of our research.
The first case is the cube $\Gamma=\Hamming(3,2)$. The standard 
construction of the cube graph is precisely the same as for the Hamming 
graphs with $d=3$ and $q=2$, and so this graph is the `standard' 
cube with $8$ vertices (the cube $\Hamming(3,2)$ is also isomorphic 
to the grid complement $\gridcomp{4}$).
The second case  is $\Gamma=\Hamming(d,2)$ when  $d>2$; see 
Lemma~\ref{lem:cube1}.

Some  platonic solids (cube, octahedron and icosahedron)   appear in some form in our investigation. The   cube appears as the $\gridcomp{4}$. We discuss in more detail the octahedron and the icosahedron.
The octahedron (see Figure \ref{fig:Octahedron}) has 6 vertices and  diameter 2. Its  automorphism group $S_2 \wr S_3$  acts imprimitively preserving the partition of vertices into antipodal pairs. We  denote by 
$\pi$  the natural projection $S_2 \wr S_3\rightarrow S_3$. 
\begin{figure}[h]
\centering

\begin{tikzpicture}

\draw (1,2.5)-- (3,1); \draw (3,1)-- (5,2.5);

\draw (1,2.5)-- (3,2); \draw (3,2)-- (5,2.5);

\draw (1,2.5)-- (3,3); \draw (3,3)-- (5,2.5);

\draw (1,2.5)-- (3,4); \draw (3,4)-- (5,2.5);

\draw (3,1)-- (3,2); \draw (3,2)-- (3,3); \draw (3,3)-- (3,4);

\draw (3,1) .. controls (3.5,2) and (3.5,3) .. (3,4);

\filldraw[white] (1,2.5) circle (2pt)  (5,2.5) circle (2pt);
\draw[black] (1,2.5) circle (2pt) node[left]{$a$} (5,2.5) circle (2pt) node[right]{$a'$};

\filldraw[white] (3,4) circle (2pt) (3,1) circle (2pt);
\draw[black] (3,4) circle (2pt)  node[above]{$b$}(3,1) circle (2pt)  node[below]{$c$};

\filldraw[white] (3,2) circle (2pt) (3,3) circle (2pt);
\draw[black] (3,2) circle (2pt)  node[xshift=-0.15cm, yshift=0.25cm]{$b'$} (3,3) circle (2pt)  node[xshift=-0.15cm, yshift=0.2cm]{$c'$};


\end{tikzpicture}
\caption{The octahedron, displayed according to its distance-partition.}\label{fig:Octahedron}
\end{figure}
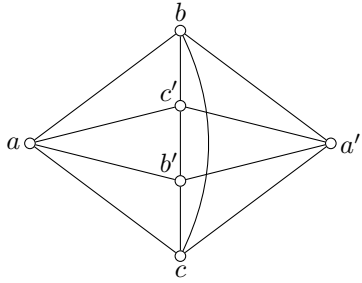

\begin{lemma}\label{lem:Octahedron}
Let $\Gamma$ be the octahedron, and let $G \leqslant \Aut\Gamma $. Then $\Gamma$ is  not $(G,2)$-arc transitive. Further,  $\Gamma$ is $(G,2)$-distance transitive, if and only if  either $G= S_2\wr S_3$, or $G$ is an index $2$ subgroup of $S_2\wr S_3$  and $G\pi =S_3$.
\end{lemma}
\begin{proof}
Since $\Gamma$ is non-complete of girth 3, $\Gamma$ is not $(G,2)$-arc
transitive. Now assume that $\Gamma$ is $(G,2)$-distance transitive.
Let $u=a$ be the vertex   in the graph of Figure \ref{fig:Octahedron}.
Since $G_u$ is transitive on $\Gamma(u)$ and $|\Gamma(u)|=4$,
$|G_u|$ is divisible by 4. Further, $|G:G_u|=6$, and  so $|G|$ is
divisible by 24. Suppose that $G$ is a proper subgroup of $\Aut
\Gamma=S_2\wr S_3$. Then $|G|=24$. As $|\Aut \Gamma|=48$, $G$ is
an index $2$ subgroup of $S_2\wr S_3$.  The three antipodal blocks of
$V\Gamma$ in the graph of Figure \ref{fig:Octahedron} are
$\Delta_1=\{a,a'\}$, $\Delta_2=\{b,b'\}$ and
$\Delta_3=\{c,c'\}$. Since $G$ is transitive on $V\Gamma$, $G$ is
transitive on the three antipodal blocks. Thus the image $G\pi$  of
$G$ in $S_3$ is $\mathbb{Z}_3$ or $S_3$. Assume
$G\pi=\mathbb{Z}_3$. Then $G_u$ acts on the three antipodal blocks
trivially. Hence $G_u$ does not map $\Delta_2$ to $\Delta_3$,
contradicting that $G_u$ is transitive on $\Gamma(u)$. Therefore
$G\pi=S_3$. Simple calculation shows that the conditions stated in the lemma
are sufficient for $2$-distance transitivity.
\end{proof}


The icosahedron  has automorphism group $S_2 \times A_5$ acting arc transitively.
\begin{lemma}\label{lem:Icosahedron}
  Let $\Gamma$ be the icosahedron, and let $G \leqslant \Aut\Gamma $. The
  graph $\Gamma$ is $(G,2)$-distance transitive if and only if
  $G = S_2 \times A_5$ or $G=A_5$. In particular, $\Gamma$ is  not $(G,2)$-arc transitive.
\end{lemma}
\begin{proof}
  By \cite[Theorem 1.5]{DJLP-prime}, $\Aut\Gamma\cong S_2 \times A_5$.   It is easy to see that for $G\in\{S_2\times A_5,A_5\}$,
  $\Gamma$ is $(G,2)$-distance
  transitive.
Suppose that $\Gamma$ is $(G,2)$-distance transitive. 
Then $G$ is transitive on $V\Gamma$ and $G_u$ is transitive on $\Gamma(u)$, and
so   $12=|V\Gamma|$ divides $|G|$ and $|\Gamma(u)|=5$ divides $|G_u|$. Thus  60 divides $|G|$. Since $G\leq \Aut\Gamma\cong S_2 \times A_5$, it follows that  $G = S_2 \times A_5$ or $G=A_5$. Finally, as $\Gamma$ is a
non-complete graph of girth 3, $\Gamma$ is  not $(G,2)$-arc transitive.
\end{proof}

%
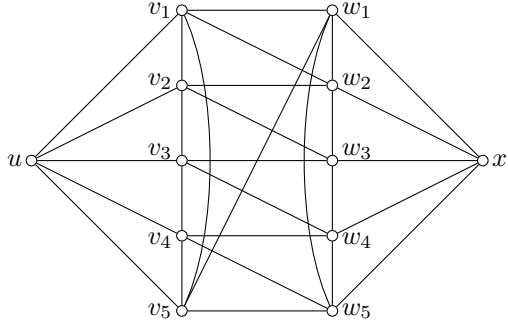
\begin{figure}[t]
\centering

\begin{tikzpicture}

\draw (-1,2) -- (1,4);
\draw (-1,2) -- (1,3);
\draw (-1,2) -- (1,2);
\draw (-1,2) -- (1,1);
\draw (-1,2) -- (1,0);

\draw (1,4) -- (3,4);
\draw (1,3) -- (3,3);
\draw (1,2) -- (3,2);
\draw (1,1) -- (3,1);
\draw (1,0) -- (3,0);

\draw (1,4) -- (3,3);
\draw (1,3) -- (3,2);
\draw (1,2) -- (3,1);
\draw (1,1) -- (3,0);
\draw (1,0) -- (3,4);

\draw (5,2) -- (3,4);
\draw (5,2) -- (3,3);
\draw (5,2) -- (3,2);
\draw (5,2) -- (3,1);
\draw (5,2) -- (3,0);

\draw (3,0) .. controls (2.5,1) and (2.5,3) .. (3,4);
\draw (3,0) -- (3,4);

\draw (1,0) .. controls (1.5,1) and (1.5,3) .. (1,4);
\draw (1,0) -- (1,4);

\filldraw[white] (1,4) circle (2pt);
\draw[black] (1,4) circle (2pt) node[left]{$v_1$};
\filldraw[white] (1,3) circle (2pt);
\draw[black] (1,3) circle (2pt) node[left, yshift = .05cm]{$v_2$};
\filldraw[white] (1,2) circle (2pt);
\draw[black] (1,2) circle (2pt) node[left, yshift = .12cm]{$v_3$};
\filldraw[white] (1,1) circle (2pt);
\draw[black] (1,1) circle (2pt) node[left, yshift = -.06cm]{$v_4$};
\filldraw[white] (1,0) circle (2pt);
\draw[black] (1,0) circle (2pt) node[left]{$v_5$};

\filldraw[white] (3,4) circle (2pt);
\draw[black] (3,4) circle (2pt) node[right]{$w_1$};
\filldraw[white] (3,3) circle (2pt);
\draw[black] (3,3) circle (2pt) node[right, yshift = .05cm]{$w_2$};
\filldraw[white] (3,2) circle (2pt);
\draw[black] (3,2) circle (2pt) node[right, yshift = .12cm]{$w_3$};
\filldraw[white] (3,1) circle (2pt);
\draw[black] (3,1) circle (2pt) node[right, yshift = -.06cm]{$w_4$};
\filldraw[white] (3,0) circle (2pt);
\draw[black] (3,0) circle (2pt) node[right]{$w_5$};

\filldraw[white] (-1,2) circle (2pt);
\draw[black] (-1,2) circle (2pt) node[left]{$u$};
\filldraw[white] (5,2) circle (2pt);
\draw[black] (5,2) circle (2pt) node[right]{$x$};

\end{tikzpicture}
\caption{The icosahedron, displayed according to its distance-partition.}\label{fig:Icosahedron}
\end{figure}

\section{Graphs of girth $4$}\label{sect:girth4}
By the assertion of Lemma \ref{lem:girth5}, to study the family of
$(G,2)$-distance transitive, but not $(G,2)$-arc transitive graphs,  we
only need to consider the graphs with girth $3$ or $4$.  This section
is devoted to the  girth  $4$ case, and the structure of
such graphs depends strongly
upon the value of the constant $c_2$ as in Definition~\ref{definition: intersectionarray}.
We begin with a simple combinatorial result:

\begin{lemma}\label{lem:sizeofgamma2}
Let $\Gamma$ be a $(G,2)$-distance transitive graph with valency $k$ and girth at least $4$. Let $u\in V\Gamma$. Then there are $k(k-1)$ edges between $\Gamma(u)$ and $\Gamma_2(u)$, and  $k(k-1) = c_2 |\Gamma_2(u)|$.
\end{lemma}
\begin{proof}
Consider a vertex $v \in \Gamma(u)$. Since $\Gamma$ has girth more than $3$, all of the neighbors of $v$, except for $u$, lie in $\Gamma_2(u)$. Thus, there are $k-1$ edges from $v$ to $\Gamma_2(u)$. Since there are $k$ such vertices $v$, there are $k(k-1)$ edges between $\Gamma(u)$ and $\Gamma_2(u)$.
As $\Gamma$ is $(G,2)$-distance transitive,  the equation $k(k-1) = c_2 |\Gamma_2(u)|$ follows by counting the same quantity from the other side: each vertex in $\Gamma_2(u)$ is incident with exactly $c_2$ edges between $\Gamma_2(u)$ and $\Gamma(u)$.
\end{proof}


For a vertex $u\in\VGamma$, we denote by $G_u^{\Gamma_i(u)}$
the permutation group induced by $G_u$ on $\Gamma_i(u)$.

\begin{lemma}\label{lem:c_2=2}
  Let $\Gamma$ be a $(G,2)$-distance transitive, but not $(G,2)$-arc transitive graph with valency $k$ and girth $4$, and suppose that $c_2=2$.
  Then $G_u$   acts $2$-homogeneously, but not $2$-transitively on
  $\Gamma(u)$ for each  $u\in \VGamma$. Further,
  $k=p^e\equiv 3 \pmod 4$ where $p$ is a prime.
\end{lemma}
\begin{proof}
Since  $c_2=2$, each vertex $w\in\Gamma_2(u)$ uniquely determines a
$2$-subset in $\Gamma(u)$, namely the intersection $\Gamma(w) \cap
\Gamma(u)$. We claim that the  map $\psi:w\mapsto\Gamma(w)\cap\Gamma(u)$
is a bijection between $\Gamma_2(u)$ and the set of 2-subsets of $\Gamma(u)$.
Suppose that $\psi(w_1)=\psi(w_2)=\{v_1,v_2\}$. Then
$u,\ w_1,\ w_2\in\Gamma(v_1)$ and $v_2\in\Gamma_2(v_1)$. On the other hand,
as $v_2$ is adjacent to $u,\ w_1,\ w_2$,
there are three edges from $v_2$ to $\Gamma(v_1)$, which is impossible,
as $c_2=2$. Hence $\psi$ is injective.
Since $\Gamma$ has girth 4, it follows from Lemma
\ref{lem:sizeofgamma2} that  $|\Gamma_2(u)| =
k(k-1)/2=\binom{k}{2}$,  and so the map $\psi$ is a bijection.
Hence $G_u$
is transitive on $\Gamma_2(u)$ if and only if it is transitive on the
set of $2$-subsets in $\Gamma(u)$, that is, $G_u^{\Gamma(u)}$ acts
$2$-homogeneously on $\Gamma(u)$. Since $\Gamma$ is not $(G,2)$-arc
transitive, $G_u^{\Gamma(u)}$ is not
$2$-transitive on $\Gamma(u)$. Thus by   Lemma \ref{2dt-2homonot2t},
$k=p^e\equiv 3 \pmod 4$ where $p$ is a prime.
\end{proof}

In the following lemma we characterize $(G,2)$-distance transitive,
but not $(G,2)$-arc transitive Hamming graphs over an alphabet of size $2$.

\begin{lemma}\label{lem:cube1}
Let $\Gamma = \Hamming(d,2)$ with $d > 2$, 
and let $G  \leqslant\Aut\Gamma \cong S_2\wr S_{d}$. 
Then $\Gamma$ is $(G,2)$-distance transitive, but not $(G,2)$-arc transitive 
if and only if $G=S_2\wr H$ where $H$ is a $2$-homogeneous,
but not $2$-transitive subgroup of $S_d$.
Further, in this case, $d=p^e\equiv 3 \pmod 4$.
\end{lemma}
\begin{proof}
  By  \cite[p.~222]{BCN}, $\Gamma$ is $\Aut \Gamma$-distance
  transitive of girth 4, valency $d$, and $c_2=2$.
 Assume that the action of $G$
  on  $\Gamma$ is 2-distance transitive, but not $2$-arc
  transitive.   Then by Lemma~\ref{lem:c_2=2},  $G_u$ is 2-homogeneous, but not $2$-transitive
  on $\Gamma(u)$, for all $u$.
  Further,
$d=p^e\equiv 3 \pmod 4$
  where $p$ is a prime.
  Let
  $A=\Aut\Gamma=M\rtimes S_d$ where $M=(S_2)^d$.
  Let $u$ be the vertex $(1,\ldots,1)$ and set $H=G_u$.
  If $g\in G$, then $g=mh$ where $m\in M$ and $h\in S_d$, and
  so $h\in H$. Hence $G\leq MH$. 
  Then, by   Dedekind's Modular Law, $(G\cap M)H=G\cap (MH)=G$.
  Thus $G\cap M$ is a transitive subgroup of $G$. Since $M$ is regular,
 $G\cap M=M$, and so $M\leq G$. Thus
  $G=M\rtimes H=S_2\wr H$. As the action of $H$ on
  $\Gamma(u)$ is faithful,  $H=G_u^{\Gamma(u)}$.

 Conversely, assume that $G=S_2\wr H$ and $H$ is a $2$-homogeneous,
 but not $2$-transitive subgroup of $S_d$. Then
$G$ is transitive on $V\Gamma$. Since $G_u^{\Gamma(u)}=G_u=H$, 
$G_u^{\Gamma(u)}$ acts $2$-homogeneously,
but not $2$-transitively on $\Gamma(u)$ for each  $u\in \VGamma$. Hence $\Gamma$ is  not $(G,2)$-arc transitive and $G_u^{\Gamma(u)}$ is transitive on the set of 2-subsets of $\Gamma(u)$. Since
$\Gamma$ has  girth 4 and $c_2=2$,
we can construct  a one-to-one correspondence between the 2-subsets of
$\Gamma(u)$ and vertices of $\Gamma_2(u)$ as in the proof
of Lemma~\ref{lem:c_2=2}.
Thus $G_u$ is transitive on  $\Gamma_2(u)$,
so $\Gamma$ is $(G,2)$-distance transitive.
\end{proof}

We have treated the case where $c_2=2$. When $c_2$ is `large' (that is, close to the valency) we can say a lot about the structure of $\Gamma$.

\begin{lemma}\label{lem:c_2=k}
  If $\Gamma$ is a connected
  $(G,2)$-distance transitive graph with valency $k$
  and girth $4$, then the following are valid.
  \begin{enumerate}
  \item If $c_2=k$, then $\Gamma = \K_{k,k}$.
    \item If  $k\geq 3$ and $c_2=k-1$, then $\Gamma = \gridcomp{(k+1)}$.
    \end{enumerate}
\end{lemma}
\begin{proof}
(i) Let $(u,v,w)$ be a $2$-arc. Since  $\Gamma$ has girth 4, $u$ and
  $w$ are nonadjacent, so  $w$ has $k$ neighbors in $\Gamma(u)$, as
  $c_2=k$. Since the valency of $\Gamma$ is $k$, this forces
  $\Gamma(u) = \Gamma(w)$. By the  $(G,2)$-distance transitivity of
  $\Gamma$,  every vertex in $\Gamma_2(u)$ has all its neighbors
in  $\Gamma(u)$, and this implies that $\Gamma_3(u)$ is empty and there
  are no edges in $\Gamma_2(u)$.  Thus  $\Gamma$ is a bipartite graph
  and  the two biparts are  $\Gamma(u)$ and $\{ u \} \cup
  \Gamma_2(u)$. Every edge between the two biparts is present, so
  $\Gamma$ is a complete bipartite graph. Since $\Gamma$ is regular of
  valency $k$, we  have  $\Gamma = \K_{k,k}$.

(ii) Let $(u,v,w)$ be a 2-arc. Since $\Gamma$ has girth 4 and
  $c_2=k-1$, by Lemma~\ref{lem:sizeofgamma2}, we have $|\Gamma_2(u)| =
  k$. 
  Let $w'$ be the unique vertex in $\Gamma_2(u)$ that is not
  adjacent to $v$.
Assume that the induced subgraph $[\Gamma_2(u)]$ contains
an edge. As $G_u$ is transitive on $\Gamma_2(u)$, every vertex of
$\Gamma_2(u)$  is adjacent to some vertex of
$\Gamma_2(u)$. Since $\Gamma$ has girth 4, the $k-1$ vertices in
$\Gamma_2(u)\cap \Gamma(v)$ are pairwise nonadjacent, so every vertex
of $\Gamma_2(u)\cap \Gamma(v)$  is adjacent to $w'$, which is
impossible, as $|\Gamma(u)\cap \Gamma(w')|=k-1$. Thus  there are no
edges in $[\Gamma_2(u)]$. Thus  each vertex in
$\Gamma_2(u)$ is adjacent to a unique vertex in $\Gamma_3(u)$.

Let $z\in \Gamma_3(u)\cap \Gamma(w)$. Since $c_2=k-1$,  every pair of
vertices at distance $2$ have $k-1$ common neighbors, so
$|\Gamma(v)\cap \Gamma(z)|=k-1$. Hence $z$ is adjacent to all vertices
of $\Gamma_2(u)$ that are adjacent to $v$.
If for all $v'\in \Gamma(u)$, $\Gamma_2(u)\cap
  \Gamma(v)=\Gamma_2(u)\cap \Gamma(v')$, then $|\Gamma_2(u)|=k-1$,
  which is a contradiction. Thus $\Gamma(u)$ contains a vertex $v'$
  such that $\Gamma_2(u)\cap \Gamma(v)\neq \Gamma_2(u)\cap
  \Gamma(v')$. In particular, $\Gamma_2(u)= \Gamma_2(u)\cap
  (\Gamma(v)\cup \Gamma(v'))$.
Now $v'$ and $z$ must have a
common neighbor in $\Gamma_2(u)$, and so $v'$ and $z$ are
at distance $2$. Thus, as $c_2=k-1$, $z$ is adjacent to
all vertices of $\Gamma_2(u)$ that are adjacent to $v'$.
Thus $z$ is adjacent to all vertices of $\Gamma_2(u)$. Since
$|\Gamma_2(u)|=k$, we find that there are no more vertices in $\Gamma$.
Therefore,  we have determined $\Gamma$
completely, and $\Gamma = \gridcomp{(k+1)}$.
\end{proof}




\section{Proof of Main Results}\label{sect:proofs}

We  first prove Theorem~\ref{thm:valency stuff}.

\begin{proof}[{\bf Proof of Theorem \ref{thm:valency stuff}}]
Since $\Gamma$ has  girth $4$, it follows that $2\leq c_2\leq k$.
If $c_2=k$, then,  by Lemma
\ref{lem:c_2=k},  $\Gamma = \K_{k,k}$. However, by Lemma
\ref{lem:complete bipartite},  $\Gamma$ is $(G,2)$-arc transitive,
whenever it is $(G,2)$-distance transitive,
and hence this case cannot arise.
Thus  $2\leq c_2\leq k-1$.  Statement~(i) now follows from
Lemmas~\ref{lem:c_2=k}(ii) and~\ref{lem:gridcomp},
while statement~(ii) follows from
Lemma~\ref{lem:c_2=2} 
\end{proof}

Next we prove Corollary~\ref{thm:primeval}.

\begin{proof}[{\bf Proof of Corollary \ref{thm:primeval}}]
If $p=2$, then $\Gamma$ is a cycle graph, so $\Gamma$ is
$(G,2)$-distance transitive if and only if it is $(G,2)$-arc
transitive, which is a contradiction. Thus $p\geq 3$. Then by Theorem
\ref{thm:valency stuff}, either $\Gamma \cong \gridcomp{(p+1)}$, or
$2\leq c_2 \leq p-2$. Assume that  $2\leq c_2 \leq p-2$. It follows
from Lemma \ref{lem:sizeofgamma2} that $p(p-1) = c_2
|\Gamma_2(u)|$. Since $2\leq c_2 \leq p-2$,   $p$ and $c_2$ are
coprime, so $c_2$ divides $p-1$. As $c_2<p-1$, we get $2\leq c_2\leq
(p-1)/2$ and this proves~(i). Statement~(ii) follows from
Theorem \ref{thm:valency stuff}(ii).
Assume that $c_2=
(p-1)/2$.  By Lemma~\ref{lem:sizeofgamma2},
$|\Gamma_2(u)|=2p$.  If $G_u$ were primitive on
$\Gamma_2(u)$, then by Lemma \ref{val-2p-1}, we would have,
$p=5$, and hence $c_2=2$. However, In this case
 $p\equiv 3 \pmod 4$, which is a contradiction.
Thus  $G_u$ is imprimitive on $\Gamma_2(u)$ and this shows~(iii).
\end{proof}

One can form  an infinite family of examples 
that satisfy the conditions of Corollary \ref{thm:primeval}
from Hamming graphs $\Hamming(p,2)$ using Lemma~\ref{lem:cube1}.

In the following, we prove Theorem \ref{thm:small val}, that is,
we determine all $(G,2)$-distance transitive, but not $(G,2)$-arc
transitive graphs of valency at most $5$. We split the proof into
two parts, as we consider the girth 4 and~3 cases separately
in Propositions~\ref{lem:valency 3} and~\ref{2dtval4-girth3}, respectively.


\begin{proposition}\label{lem:valency 3}
  Let $\Gamma$ be a connected $(G,2)$-distance transitive, but not
  $(G,2)$-arc transitive graph of girth $4$ and valency $k\in\{3,4,5\}$. Then
    $\Gamma \cong \gridcomp{k+1}$, and $G$ satisfies
    Condition \ref{gridcondition}.
\end{proposition}
\begin{proof}
  We claim that $c_2=k-1$ in all cases.
  By Theorem \ref{thm:valency stuff},
  $c_2\leq k-1$. If $k=3$, then $c_2\geq 2=k-1$
  follows from the girth condition, and so $c_2=k-1$.
  If $k\in\{4,5\}$ and $c_2\leq k-2$, then  we must have that $c_2=2$ (use Corollary~\ref{thm:primeval}
    for $k=5$). Hence, by Lemma \ref{lem:c_2=2},   $k \equiv 3 \pmod{4}$:
    a contradiction, as $k\in\{4,5\}$. Now the rest follows from
    Theorem~\ref{thm:valency stuff}(i).
\end{proof}



\begin{proposition}\label{2dtval4-girth3}
  Let $\Gamma$ be a connected
  $(G,2)$-distance transitive graph of girth $3$
and valency $4$ or $5$,
  and let $u\in V\Gamma$. Then one of the following is valid.
  \begin{enumerate}
\item $\Gamma$ is the octahedron and either  $G=S_2\wr S_3$ or $G$ is an index $2$ subgroup of $S_2\wr S_3$  and $G$ projects onto $S_3$;
\item $\Gamma\cong \Hamming(2,3)$ and either  $G=S_3\wr S_2$ or $G$ is an index $2$ subgroup of $S_3\wr S_2$  and $G$ projects onto $S_2$;
  \item $|\Gamma_2(u)|=8$ and $\Gamma$ is the line graph of a connected cubic $(G,3)$-arc transitive graph;
\item $\Gamma$ is the icosahedron and  $G = A_5$ or $A_5 \times S_2$.
  \end{enumerate}
  In cases (i)--(iii), the valency of $\Gamma$ is $4$, while
  in case~(iv), the valency is $5$.
\end{proposition}

\begin{proof}
  Suppose first that the valency is~4.  Since $\Gamma$ is
  $(G,2)$-distance transitive of valency  $4$ and girth $3$, it
  follows that the induced graph
  $[\Gamma(u)]$ is a vertex transitive graph with 4
  vertices of  valency $k$ where $1\leq k\leq 3$.  If $[\Gamma(u)]$
  has valency 3, then $[\Gamma(u)]$ is complete, and so $\Gamma$ is
  complete, which is a contradiction.  If $[\Gamma(u)]$ has valency 2, then
  $[\Gamma(u)]\cong C_4$.  Hence  $|\Gamma_2(u)\cap \Gamma(v)|=1$ for
  any arc $(u,v)$,  so $G_{u,v}$ is transitive on $\Gamma_2(u)\cap
  \Gamma(v)$, that is,  $\Gamma$ is $(G,2)$-geodesic transitive. Thus
  by   \cite[Corollary 1.4]{DJLP-2}, $\Gamma$ is the octahedron. It
  follows from Lemma \ref{lem:Octahedron} that either $G= S_2\wr S_3$,
  or $G$ is an index 2 subgroup of $S_2\wr S_3$ and $G$ projects onto
  $S_3$. Hence, case~(i) is valid.

Now suppose that  $[\Gamma(u)]$ has valency 1. Then $[\Gamma(u)]\cong
2\K_2$ and there are 8 edges between
$\Gamma(u)$ and $\Gamma_2(u)$. Further,   each arc lies in a unique
triangle. Let  $\Gamma(u)=\{v_1,v_2,v_3,v_4\}$ be such that
$(v_1,v_2)$ and $(v_3,v_4)$ are two arcs.  Then $|\Gamma_2(u)\cap
\Gamma(v_1)|=2$, say $\Gamma_2(u)\cap \Gamma(v_1)=\{w_1,w_2\}$.  Since
$[\Gamma(v_1)]\cong 2\K_2$, it follows that $v_2$ is adjacent to
neither $w_1$ nor $w_2$.
As $|\Gamma_2(u)\cap \Gamma(v_2)|=2$, we have $|\Gamma_2(u)|\geq 4$.
Since there are 8 edges
between $\Gamma(u)$ and $\Gamma_2(u)$ and since $G_u$ is transitive
on $\Gamma_2(u)$, we obtain that
$8\mid |\Gamma_2(u)|$, and so $|\Gamma_2(u)|\in\{4,8\}$.

Suppose first that $|\Gamma_2(u)|=4$.
As noted above, $v_2$ is not adjacent to $w_1$ or $w_2$.
Set $\Gamma_2(u)\cap\Gamma(v_2)=\{w_3,w_4\}$.
Then $\Gamma_2(u)=\{w_1,w_2,w_3,w_4\}$. 
 Since $[\Gamma(v_1)]\cong [\Gamma(v_2)]
  \cong 2\K_2$, it follows that $w_1,w_2$ are adjacent and, similarly, $w_3,w_4$
  are adjacent.
  Since $|\Gamma_2(u)|=4$ and there
  are 8 edges between $\Gamma(u)$ and $\Gamma_2(u)$, we
  must have $|\Gamma(u)\cap \Gamma(w_i)|=2$.  Since $v_2,w_1$ are nonadjacent,
$w_1$ is adjacent either to $v_3$ or to $v_4$, say $v_3$.  Then
$\Gamma(u)\cap \Gamma(w_1)=\{v_1,v_3\}$. As each arc lies in a unique
triangle and $(v_1,w_1,w_2)$ is a triangle, it follows that $v_3$ is
not adjacent to $w_2$. Hence   $v_3$ is adjacent to either
$w_3$ or $w_4$, say $w_3$. Then $\Gamma(v_3)=\{u,v_4,w_1,w_3\}$. Since
$[\Gamma(v_3)]\cong 2\K_2$ and $u,v_4$ are adjacent, it follows that
$w_1,w_3$ are adjacent. Thus,
$\Gamma(w_1)=\{v_1,w_2,v_3,w_3\}$. Finally, as $|\Gamma_2(u)\cap
\Gamma(v_4)|=2$ and $v_4$ is  adjacent to neither $w_1$ nor $w_3$,
$v_4$ is adjacent to both $w_2$ and $w_4$. Since $[\Gamma(v_4)] \cong
2\K_2$ and $(v_3,u,v_4)$ is a triangle, it follows that  $w_2$, $w_4$
are  adjacent. Now, the graph $\Gamma$ is completely determined
and  $\Gamma\cong \Hamming(2,3)$.
By \cite[Theorem 9.2.1]{BCN},  $\Gamma$ is $(\Aut\Gamma,2)$-distance
transitive where $\Aut\Gamma\cong S_3\wr S_2$. Suppose that $G$ is a
proper subgroup of $\Aut\Gamma$.  Since $G_u$ is transitive on
$\Gamma(u)$ and $|\Gamma(u)|=4$,  $|G_u|$ is divisible by 4, so $|G|$
is divisible by $4|\VGamma|=36$. It follows that
$|G|=36$, so $G$ is an index $2$ subgroup of $S_3\wr S_2$. Finally, as
$G_u$ is transitive on $\Gamma(u)$, $G_u$ projects onto $S_2$.
Thus~(ii) is valid.

Let us now consider the case when  $|\Gamma_2(u)|=8$.  Then for
each $z\in \Gamma_2(u)$, there is a unique 2-geodesic between $u$ and
$z$. Hence there is a one-to-one correspondence between the set of
2-geodesics starting from $u$ and the set of vertices in
$\Gamma_2(u)$.  Since $G_u$ is transitive on $\Gamma_2(u)$, it follows
that $G_u$ is transitive on the set of 2-geodesics starting from $u$,
so $\Gamma$ is $(G,2)$-geodesic transitive. Therefore by  Lemma
\ref{2gt-val4}, $\Gamma$ is the line graph of a connected cubic
$(G,3)$-arc transitive graph. Therefore~(iii) is valid.


Assume now that the valency is 5. Let $(u,v)$ be an arc. Since
$\Gamma$ is $G$-arc transitive, the induced subgraph $[\Gamma(u)]$ is
vertex transitive.   As $\Gamma$ has girth 3 and non-complete, the
valency $k$ of $[\Gamma(u)]$ is   at most $3$.  Since  $[\Gamma(u)]$
is  undirected, it follows that  $[\Gamma(u)]$ has $5k/2$ edges, and
so $k$ is even; that is, $k=2$.  Thus  $[\Gamma(u)]\cong C_5$.

Set  $\Gamma(u)=\{v_1,v_2,v_3,v_4,v_5\}$ with $v_1=v$ and assume
$(v_1,\ldots,v_5)$ is a 5-cycle.  Then $|\Gamma_2(u)\cap
\Gamma(v_1)|=2$ and say $\Gamma_2(u)\cap
\Gamma(v_1)=\{w_1,w_2\}$. Then $\Gamma(v_1)=\{u,v_2,v_5,w_1,w_2\}$. As
$[\Gamma(v_1)]\cong C_5$ and $(v_2,u,v_5)$ is a 2-arc, it follows that
$w_1,w_2$ are adjacent, $v_2$ is adjacent to one of $w_1$ and $w_2$ and
$v_5$ is adjacent to the other. Without loss of generality, assume
$v_2$ is adjacent to  $w_1$ and $v_5$ is adjacent to  $w_2$. In
particular, $v_2$ and  $w_2$ are not adjacent.  Moreover,  $2\leq
c_2\leq 4$. Since there are 10 edges  between $\Gamma(u)$ and $
\Gamma_2(u)$, we have $10=c_2|\Gamma_2(u)|$,   so $c_2=2$ and
$|\Gamma_2(u)|=5$.

Since
$|\Gamma_2(u)\cap \Gamma(v_2)|=2$, there exists $w_3$ in $\Gamma_2(u)$ which is adjacent to $v_2$, and so
$\Gamma(v_2)=\{u,v_1,v_3,w_1,w_3\}$. Note that $(w_1,v_1,u,v_3)$ is a 3-arc, and as $[\Gamma(v_2)]\cong C_5$, it follows that $w_3$ is adjacent to both $v_3$ and $w_1$.
Since $G_u$ is transitive on $\Gamma_2(u)$, $[\Gamma_2(u)]$ is a vertex transitive graph. Recall that  $w_1$ is adjacent to $w_2$ and $w_3$. It follows that
$[\Gamma_2(u)]\cong C_5$. Thus $|\Gamma_3(u)\cap \Gamma(w_1)|=1$, say $\Gamma_3(u)\cap \Gamma(w_1)=\{e\}$. Then $(v_1,w_1,e)$ and $(v_2,w_1,e)$ are two 2-geodesics. As $c_2=2$, $|\Gamma(v_1)\cap \Gamma(e)|=|\Gamma(v_2)\cap \Gamma(e)|=2$. Hence $\{w_1,w_2,w_3\}\subseteq \Gamma_2(u)\cap \Gamma(e)$.

Since $|\Gamma_2(u)\cap \Gamma(v_3)|=2$, there exists $w_4(\neq w_3)\in \Gamma_2(u)$ such that $v_3,w_4$ are adjacent.
Noting that $\Gamma(u)\cap \Gamma(w_1)=\{v_1,v_2\}$ and $\Gamma(u)\cap \Gamma(w_2)=\{v_1,v_5\}$, we find  $w_4\notin \{w_1,w_2,w_3\}$. Since
$[\Gamma(v_3)]\cong C_5$ and $(w_3,v_2,u,v_4)$ is a 3-arc, it follows that $w_4$ is adjacent to both $v_4$ and $w_3$.
As
$(v_3,w_3,e)$ is a 2-geodesic, $|\Gamma(v_3)\cap \Gamma(e)|=2$, so $w_4\in \Gamma_2(u)\cap \Gamma(e)$.
Now  $(v_4,w_4,e)$ is a 2-geodesic, so $|\Gamma(v_4)\cap \Gamma(e)|=2$, hence $\Gamma_2(u)\cap \Gamma(v_4)\subset  \Gamma(e)$.
Let the remaining vertex of $\Gamma_2(u)$ be $w_5$. Since $|\Gamma(u)\cap \Gamma(w_5)|=2$, it follows that $w_5$ is adjacent to both $v_4,v_5$. Hence
$\Gamma_2(u)\cap \Gamma(v_4)=\{w_4,w_5\}\subset   \Gamma(e)$. Thus $\Gamma_2(u)=\Gamma(e)$, so $\Gamma_3(u)=\{e\}$.
Now we have completely determined the graph $\Gamma$, and this graph is the
icosahedron.  Finally, by Lemma \ref{lem:Icosahedron},  $G\cong S_2\times A_5$ or $A_5$.
\end{proof}

\begin{proof}[The proof of Theorem~\ref{thm:small val}]
  If the valency of $\Gamma$ is 2 or the girth is greater than 4,
  then $\Gamma$
    cannot be $(G,2)$-distance transitive,
    but not $(G,2)$-arc transitive.
    Hence the valency is at least 3.
    If the valency and the girth are both equal to
    3, then $\Gamma=\K_4$.
    Hence Theorem~\ref{thm:small val} follows from
    Proposition~\ref{lem:valency 3}, in the case of girth 4,
    and from Proposition~\ref{2dtval4-girth3} in the case of
    girth~4.
\end{proof}

\end{document}